  \documentclass[12pt]{amsart}
\address{Max-Planck-Institut f\"ur Mathematik in den Naturwissenschaften. Inselstraße 22, 04103 Leipzig, Germany.}

  \email{vasirog[at]gmail.com}

\usepackage[utf8]{inputenc}
\usepackage[T2A]{fontenc}
\usepackage{enumerate, amssymb, amsmath, amsthm}
\usepackage{amscd}
\usepackage{graphicx}

\usepackage[colorlinks=true,citecolor=blue, urlcolor=blue, linkcolor=blue]{hyperref}

\usepackage[a4paper,hmargin=2.5cm,vmargin=2.5cm]{geometry}
\usepackage{marginnote}
\usepackage[english]{babel}
\graphicspath{ {images/} }
\usepackage{wrapfig}
\usepackage[matrix,arrow,curve]{xy}\usepackage{mathabx}
\usepackage[OT2, T1]{fontenc}
\usepackage{appendix}
\usepackage{chngcntr}
\usepackage{etoolbox}
\usepackage{lipsum}

\DeclareSymbolFont{cyrletters}{OT2}{wncyr}{m}{n}
\DeclareMathSymbol{\Sha}{\mathalpha}{cyrletters}{"58}

\usepackage[a4paper,hmargin=2.5cm,vmargin=2.5cm]{geometry}
\usepackage[english]{babel}
\graphicspath{ {images/} }
\usepackage{wrapfig}
\usepackage[matrix,arrow,curve]{xy}\usepackage{mathabx}

\begin{document}
\renewcommand{\refname}{Bibliography}
\newtheorem{prop}{Proposition}[section]
\newtheorem{thrm}[prop]{Theorem}
\newtheorem{lemma}[prop]{Lemma}
\newtheorem{cor}[prop]{Corollary}
\newtheorem{mainthm}{Theorem}
\newtheorem{maincor}[mainthm]{Corollary}
\theoremstyle{definition}
\newtheorem{df}{Definition}
\newtheorem{ex}{Example}
\newtheorem{rmk}{Remark}
\newtheorem{conj}{Conjecture}
\newtheorem{cl}{Claim}
\newtheorem{q}{Question}
\newtheorem{constr}{Construction}
\renewcommand{\proofname}{\textnormal{\textbf{Proof:  }}}
\renewcommand{\refname}{Bibliography}
\renewcommand{\themainthm}{\Alph{mainthm}}
\renewcommand{\themaincor}{\Alph{maincor}}
\newtheorem*{prop*}{Proposition}
\newtheorem*{theorem*}{Theorem}

\renewcommand{\phi}{\varphi}
\renewcommand{\epsilon}{\varepsilon}

\renewcommand{\C}{\mathbb C}
\newcommand{\Z}{\mathbb Z}
\newcommand{\Q}{\mathbb Q}
\newcommand{\R}{\mathbb R}
\newcommand{\N}{\mathbb N}
\newcommand{\Fp}{\mathbb{F}_p}
\newcommand{\Fq}{\mathbb{F}_q}

\renewcommand{\O}{\mathcal O}
\newcommand{\g}{\mathfrak g}
\newcommand{\h}{\mathfrak h}
\newcommand{\E}{\mathcal E}
\newcommand{\F}{\mathcal F}
\newcommand{\m}{\mathfrak{m}}

\renewcommand{\i}{\sqrt{-1}}
\renewcommand{\o}{\otimes}
\newcommand{\di}{\partial}
\newcommand{\acts}{\lefttorightarrow}
\newcommand{\dibar}{\overline{\partial}}
\newcommand{\im}{\operatorname{im}}
\renewcommand{\ker}{\operatorname{ker}}
\newcommand{\Hom}{\operatorname{Hom}}
\newcommand{\tr}{\operatorname{tr}}
\newcommand{\codim}{\operatorname{codim}}
\newcommand{\rk}{\operatorname{rk}}
\newcommand{\nilp}{\operatorname{nilp}}
\newcommand{\hdot}{{\:\raisebox{3pt}{\text{\circle*{1.5}}}}}
\newcommand{\Supp}{\operatorname{Supp}}
\newcommand{\Alb}{\operatorname{Alb}}
\newcommand{\alb}{\operatorname{alb}}
\newcommand{\Hilb}{\operatorname{Hilb}}
\newcommand{\Sh}{\operatorname{Sh}}
\newcommand{\sh}{\operatorname{sh}}
\newcommand{\CP}{\mathbb{C}\mathbf{P}}
\newcommand{\Isom}{\operatorname{Isom}}
\newcommand{\Sym}{\operatorname{Sym}}
\newcommand{\Stab}{\operatorname{Stab}}
\newcommand{\Aut}{\operatorname{Aut}}
\newcommand{\sslash}{\mathbin{/\mkern-6mu/}}
\newcommand{\Pic}{\operatorname{Pic}}
\newcommand{\V}{\mathbb{V}}
\newcommand{\alg}{\operatorname{alg}}
\newcommand{\Ext}{\operatorname{Ext}}
\newcommand{\MHS}{\operatorname{MHS}}
\newcommand{\HS}{\operatorname{HS}}
\newcommand{\Gr}{\operatorname{Gr}}
\newcommand{\End}{\operatorname{End}}
\newcommand{\odef}{\operatorname{def}}
\newcommand{\ad}{\operatorname{ad}}
\newcommand{\Ad}{\operatorname{Ad}}
\newcommand{\cd}{\operatorname{cd}}
\newcommand{\an}{\operatorname{an}}
\newcommand{\orb}{\operatorname{orb}}
\newcommand{\Trop}{\operatorname{Trop}}
\newcommand{\Hdg}{\operatorname{Hdg}}

\newcommand{\GL}{\operatorname{GL}}
\newcommand{\SL}{\operatorname{SL}}
\newcommand{\SU}{\operatorname{SU}}
\renewcommand{\U}{\operatorname{U}}
\newcommand{\SO}{\operatorname{SO}}
\newcommand{\Ogr}{\operatorname{O}}
\newcommand{\Sp}{\operatorname{Sp}}

\newcommand{\gl}{\mathfrak{gl}}
\renewcommand{\sl}{\mathfrak{sl}}
\newcommand{\su}{\mathfrak{su}}
\renewcommand{\u}{\mathfrak{u}}
\newcommand{\so}{\mathfrak{so}}
\renewcommand{\sp}{\mathfrak{sp}}
\renewcommand{\g}{\mathfrak{g}}
\renewcommand{\h}{\mathfrak{h}}
\newcommand{\z}{\mathfrak{z}}

\newcommand{\G}{\mathcal{G}}

\title[Topology of higher Albanese maps]{Topology of higher Albanese maps and aspherical varieties with nilpotent fundamental group}

\author{Vasily Rogov}

\begin{abstract}Let $X$ be a normal complex algebraic variety. Let 
$\mathcal{G}^s_{\mathbb{Z}}(X)$ be the maximal torsion free nilpotent quotient of $\pi_1(X^{\an})$ of nilpotency class at most $s$. Let $F^{\bullet}\mathfrak{g}^s$ be the Morgan--Hain
Hodge filtration on the Lie algebra of the $s$-th lower central quotient of the complex Maltsev completion
of $\pi_1(X)$. We show that the natural map
$H^k(\mathcal{G}^s_{\mathbb{Z}}(X), \mathbb{Z}) \to H^k(X, \mathbb{Z})$
vanishes for $k > \dim F^1\mathfrak{g}^s$. If $\G^s_{\Z}(X)$ is of nilpotency class greater than two, this includes
the top nonvanishing degree of $H^{\bullet}(\mathcal{G}^s_{\mathbb{Z}}(X),
\mathbb{Z})$. We deduce that if the fundamental group of an aspherical
normal variety is virtually nilpotent, it is virtually two-step
nilpotent, confirming a classical conjecture in this case. This gives a positive answer to a question of Aguilar and Campana in this case of aspherical varieties. The ingredients of the proof are the $q$-convexity of higher
Albanese manifolds and the definability of higher Albanese maps.
\end{abstract}

\maketitle


\section{Introduction}

A famous paper \cite{DGMS} proves that compact K\"ahler manifolds are cohomologically formal.

Simplest examples of non-formal topological spaces are provided by nilmanifolds, which are classifying spaces for finitely presented torsion-free nilpotent groups: the only case when they are formal is when the corresponding nilpotent group is abelian \cite{Has}. This leaves one wondering whether there exist compact K\"ahler manifolds with nilpotent, non-abelian fundamental group. If $X$ is such a manifold, it needs to have sufficiently complicated higher homotopy groups to kill the Massey products coming from the map $X \to K(\pi_1(X),1)$. In particular, $X$ cannot be aspherical.

It turns out that there exist compact K\"ahler manifolds whose fundamental group is torsion-free, non-abelian, $2$-step nilpotent \cite{SvdV, Camp} (see also \cite{GM} for potential new examples).
Passing to quasi-projective varieties, one finds more examples of varieties with torsion-free nilpotent non-abelian fundamental groups (\cite[Example 4.25]{CDY}, \cite{Coll}). In all these examples, the fundamental group is still at most $2$-step nilpotent. This motivates the following conjecture, formulated by Campana in \cite{Camp} for the compact K\"ahler and smooth projective cases, and asked as a question in \cite{AC} in general.

\begin{conj}\label{at most two conj}
Let $X$ be a normal complex algebraic variety. Suppose that $\pi_1(X)$ is virtually nilpotent. Then it is virtually at most $2$-step nilpotent.
\end{conj}

This conjecture is known to hold in several special cases \cite{Camp, Rog, Shim1} including major recent progress establishing it for weakly special varieties \cite{CDHP}. As we explain above, if $X$ is aspherical and smooth projective, the only way $\pi_1(X)$ can be nilpotent is if it is abelian, so Conjecture \ref{at most two conj} is automatic in this setting. However, for aspherical quasi-projective varieties the same argument breaks down, since quasi-projective varieties may fail to be formal. Moreover, \cite[Example 4.25]{CDY} provides an example of an aspherical quasi-projective variety with nilpotent non-abelian fundamental group; cf. \cite[Example 13.1]{Suc}.

In this paper we verify Conjecture \ref{at most two conj} for aspherical varieties, without any projectivity assumption.

\begin{prop*}[Corollary \ref{aspherical cor}]
Let $X$ be an aspherical normal algebraic variety. Suppose that $\pi_1(X)$ is virtually nilpotent. Then it is virtually $2$-step nilpotent.
\end{prop*}

This result is a special case of a more general theorem, for which we first fix some notation.
Let $\Gamma$ be a finitely presented group with lower central series $\ldots \subseteq \Gamma_s \subseteq \ldots \subseteq \Gamma_1 \subseteq \Gamma$, so that $\Gamma^s:=\Gamma/\Gamma_s$ is a finitely generated nilpotent group. It admits a torsion-free nilpotent quotient $\Gamma^s \to \mathcal{G}^s_{\mathbb{Z}}$ with finite kernel (the notation $\mathcal{G}^s_{\mathbb{Z}}$ is motivated in Subsection~\ref{nilpotent basic subsec}), which induces a natural map on cohomology $H^{\bullet}(\mathcal{G}^s_{\mathbb{Z}}, \mathbb{Z}) \to H^{\bullet}(\Gamma, \mathbb{Z})$.

Now let $\Gamma = \pi_1(X)$ for a topological space $X$, and write $\mathcal{G}^s_{\mathbb{Z}}(X) := \mathcal{G}^s_{\mathbb{Z}}$. Composing with $H^{\bullet}(\pi_1(X), \mathbb{Z}) \to H^{\bullet}(X, \mathbb{Z})$ gives a map
\[
\alpha^{\bullet} \colon H^{\bullet}(\mathcal{G}^s_{\mathbb{Z}}(X), \mathbb{Z}) \to H^{\bullet}(X, \mathbb{Z}).
\]
When $X$ is a normal algebraic variety, $\mathcal{G}^s_{\mathbb{Z}}(X)$ is realized as a Zariski-dense discrete subgroup of a nilpotent complex Lie group $\mathcal{G}^s_{\mathbb{C}}(X)$, whose Lie algebra carries a canonical Hodge filtration $F^{\bullet}\mathfrak{g}^s(X)$ (see Subsection \ref{higher albanese subsec} for more details).

\begin{theorem*}[Theorem \ref{main thrm}]
Let $X$ be a normal algebraic variety and $s>0$. Then:
\begin{enumerate}
\item\label{first item intro} the map $\alpha^k \colon H^k(\mathcal{G}^s_{\mathbb{Z}}(X), \mathbb{Z}) \to H^k(X, \mathbb{Z})$ vanishes for $k>\dim F^1\mathfrak{g}^s$;
\item\label{second item intro} either $\mathcal{G}^s_{\mathbb{Z}}(X)$ has nilpotency class at most $2$, or there exists $k$ such that $\alpha^k=0$ and $H^k(\mathcal{G}^s_{\mathbb{Z}}(X), \mathbb{Z}) \neq 0$.
\end{enumerate}
\end{theorem*}
In a way, item (\ref{first item intro}) establishes a new restriction on algebraic topology of normal complex algebraic varieties, and item (\ref{second item intro}) guarantees that this restriction is indeed non-empty in many cases. The condition that $\mathcal{G}^s_{\mathbb{Z}}(X)$ has nilpotency class at most $2$ appearing in item (\ref{second item intro}) means one of two things: either $s\in\{1,2\}$, or $s>2$ and the lower central series of $\pi_1(X)$ rationally stabilize at the second step. In particular, if $\pi_1(X)$ is nilpotent and torsion-free, this means it is at most $2$-step nilpotent.

This paper is a logical continuation of our earlier work \cite{Rog}, although the reader does not need to be familiar with it. The definability of higher Albanese maps in some o-minimal structure plays a crucial role in the proof, but is entirely black-boxed (see Proposition~\ref{closure def prop}); no prior exposure to the underlying notions is needed.

The scheme of proof is very simple. We show that the $s$-th higher Albanese manifold $\Alb^s(X)$ is holomorphically $q$-complete for $q=\dim F^0\mathfrak{g}^s$. This property is inherited by $Y$, the closure of the higher Albanese image. Definability of higher Albanese maps guarantees that $\dim Y$ equals the dimension of the image itself, which, by Griffiths transversality, is at most $\dim F^1\mathfrak{g}^s-\dim F^0\mathfrak{g}^s$. Using Hamm's theorem on the homotopy type of $q$-complete spaces, we then bound the cohomological dimension of $Y$.

\textbf{Acknowledgements.} I am thankful to MPI MiS, Leipzig, for great working conditions. 
\\

\textbf{AI Usage Disclosure.} Claude Sonnet 5 (Anthropic) was used to improve style, grammar, and organization, and to assist in developing several motivating examples; the author takes full responsibility for the content of the paper.

\section{Higher Albanese manifolds}
\subsection{Nilpotent quotients}\label{nilpotent basic subsec} 
We cover some basic material on nilpotent groups, lower central series and Maltsev completions. Basic references for this are \cite{Mal, Quil}, see also surveys in \cite[Appendix A]{ABCKT} and \cite[Section 2]{Merk}.

Let $\Gamma$ be finitely presented group. Its lower central series is the tower of subgroups $\Gamma_0=\Gamma, \ \Gamma_j:=[\Gamma_{j-1}, \Gamma]$. Each  $\Gamma_j$ is normal both in $\Gamma_{j-1}$ and in $\Gamma$ and the quotients $\Gamma_{j-1}/\Gamma_j$ are abelian.

A group $\Delta$ is called nilpotent if $\Delta_s=\{e\}$ for some $s>0$. The number 
\[
\nilp(\Delta):=\min\{s \ | \Delta_s = \{e\}\}
\]
is called the \emph{nilpotency class} of $\Gamma$.

The \emph{rank} (or \emph{Hirsch length}) of a finitely presented nilpotent group $\Gamma$ is defined as
\[
\rk \Delta:=\sum \rk(\Delta_{j-1}/\Delta_j).
\]

If $\Delta$ is nilpotent and finitely presented, then its torsion is a finite normal subgroup and it admits a torsion free finite index subgroup.

Subgroups and quotients of nilpotent groups are nilpotent of non-higher nilpotency class.

Now let $\Gamma$ be an arbitrary finitely presented group. Set $\Gamma^s:=\Gamma/\Gamma_s$. For $s=1$ this is precisely the abelianisation $\Gamma^1=\Gamma/[\Gamma, \Gamma]$. In general, this is a nilpotent group of nilpotency at most $s$. In fact, exactly one of the following holds:
\begin{itemize}
\item $\nilp(\Gamma^s)=s$;
\item the lower central series stabilise at some step $s'< s$, i.e. $\Gamma_{s'}=\Gamma_{s'+1}=\ldots =\Gamma_s$. In this case $\Gamma^s\simeq \Gamma^{s'}$ and $\nilp(\Gamma^s)=s'$.
\end{itemize}

In both cases, it is true that for any $s$-step nilpotent group $\Delta$ a homomorphism $\phi \colon \Gamma \to \Delta$ uniquely factors through $\Gamma^s$.

If $\Gamma$ is abelian group, one can construct its universal representation over a field $\Bbbk$ via mapping $\Gamma \to \Gamma \o_{\Z} \Bbbk$. It turns out, that similar construction exists for arbitrary finitely generated nilpotent groups. Namely, the following result is due to Maltsev for $\Bbbk=\R$ \cite{Mal} and in full generality probably belongs to Quillen \cite{Quil}.

\begin{thrm}[Maltsev, Quillen]
Let $\Delta$ be finitely generated nilpotent group and $\Bbbk$ a field of characteristic zero. There exists a connected simply connected unipotent algebraic group over $\Bbbk$, denoted by $\G_{\Bbbk}(\Delta)$, and a representation with Zariski dense image $\mu_{\Delta, \Bbbk} \colon \Gamma \to \G_{\Bbbk}(\Delta)(\Bbbk)$ such that the following holds: for any algebraic group $\mathbf{G}$ over $\Bbbk$ and a  representation $\rho \colon \Delta \to \mathbf{G}(\Bbbk)$ there exists unique factorisation
\[
\xymatrix{
\Delta \ar[rr]^{\rho} \ar[rd]_{\mu_{\Gamma, \Bbbk}} && \mathbf{G}(\Bbbk)\\
&\G_{\Bbbk}(\Delta)(\Bbbk) \ar[ru]_{\rho'}&
}
\]
where $\rho'$ is induced by a $\Bbbk$-morphism of algebraic groups $\G_{\Bbbk}(\Delta) \to \mathbf{G}$. 
\end{thrm}

For a finitely presented group $\Gamma$ we can denote $\G^s_{\Bbbk}(\Gamma):=\G_{\Bbbk}(\Gamma^s)$ and  $\mu^s_{\Bbbk}:=\mu_{\Gamma^s, \Bbbk} \circ \pi^s$ where $\pi^s \colon \Gamma \to \Gamma^s$ is the natural projection. Then $\mu^s_{\Bbbk} \colon \Gamma \to \G_{\Bbbk}(\Gamma)(\Bbbk)$ is the universal $s$-step unipotent representation of $\Gamma$ over $\Bbbk$.

The central series $\ldots \Gamma_s \unlhd \ldots \Gamma_2 \unlhd \Gamma_1 \unlhd \Gamma$ leads to a tower of central group extensions
\[
1 \xleftarrow{} \Gamma^1 \xleftarrow{} \Gamma^2 \xleftarrow{} \ldots \xleftarrow{}\Gamma^s \xleftarrow{} \ldots
\]
that maps to a tower of central $\Bbbk$-algebraic group extensions
\begin{equation}\label{Maltsev tower}
1 \xleftarrow{} \G^1_{\Bbbk}(\Gamma) \xleftarrow{} \G^2_{\Bbbk}(\Gamma) \xleftarrow{} \ldots \xleftarrow{} \G^s_{\Bbbk}(\Gamma) \xleftarrow{} \ldots
\end{equation}
The map $\G^s_{\Bbbk}(\Gamma) \to \G^{s-1}_{\Bbbk}(\Gamma)$ is a central extension via the additive group of the vector space $Z^s \o \Bbbk$, where 
\[
Z^s:=\Gamma_{s-1}/\Gamma_s=\ker(\Gamma^s \to \Gamma^{s-1})
\]
(recall that $Z^s$ is a finitely generated abelian group).

The limit $\varprojlim \G^s_{\Bbbk}(\Gamma) = : \G_{\Bbbk}(\Gamma)$ is called the \emph{Maltsev completion of $\Gamma$ over $\Bbbk$}. This is a proalgebraic unipotent group that comes with a Zariski dense representation $\Gamma \to \G_{\Bbbk}(\Gamma)$ that is universal unipotent representation of $\Gamma$ over $\Bbbk$. A group is called \emph{rationally nilpotent} if the tower (\ref{Maltsev tower}) stabilises at a finite step (for $\Bbbk=\Q$; a posteriori this property does not depend on the field). Equivalently, $\G_{\Bbbk}(\Gamma)$ has finite nilpotency. This property is strictly weaker than being nilpotent.

Below we collect some basic features of the discussed construction.

\begin{prop}\label{completion basic prop}
Let $\Delta$ be a nilpotent group. The following holds.
\begin{enumerate}
\item $\dim_{\Bbbk}\G_{\Bbbk}(\Delta)=\rk \Delta$;
\item\label{extension} if $\Bbbk \hookrightarrow \Bbbk'$ is a field extension then $\G_{\Bbbk'}(\Delta)=\G_{\Bbbk}(\Delta) \o \Bbbk'$. In particular, $\G_{\Bbbk}(\Delta)=\G_{\Q}(\Delta) \o \Bbbk$;
\item the kernel of $\mu_{\Delta, \Bbbk}$ is precisely the torsion subgroup of $\Delta$;
\item\label{real completion} if $\Bbbk=\R$ the image of $\mu_{\Delta, \R}$ is a cocompact lattice in $\G_{\R}(\Delta)$.
\end{enumerate}
\end{prop}

We denote $\G_{\Z}(\Delta):=\im \mu_{\Delta, \Q}$ for a nilpotent group $\Delta$ and $\G_{\Z}^s(\Gamma):=\im \mu^s_{\Bbbk}$ for an arbitrary group $\Gamma$. Thanks to item (\ref{extension}) of Proposition \ref{completion basic prop}, we can identify $\im \mu^s_{ \Bbbk}$ for any field $\Bbbk$ with $\G^s_{\Z}(\Gamma)$ via  the embedding $\G^s_{\Q}(\Gamma)(\Q) \hookrightarrow \G^s_{\Bbbk}(\Gamma)(\Bbbk)$.

The discussed construction is functorial: if $\phi \colon \Gamma \to \Gamma'$ is a group homomorphism, it induces a morphism of $\Bbbk$-groups $\G^s_{\Bbbk}(\Gamma) \to \G^s_{\Bbbk}(\Gamma')$  that sends $\G^s_{\Z}(\Gamma)$ to $\G^s_{\Z}(\Gamma')$.

Item (\ref{real completion}) of Propositon \ref{completion basic prop} implies strong consequences on cohomology of torsion free nilpotent groups.
\begin{prop}\label{cohomological dimension prop}
Let $\Delta$ be a finitely presented torsion free nilpotent group of rank $r$. Then $H^r(\Delta, \Z)=\Z$ and $H^k(\Delta, \Z)=0$ for $k>r$.
\end{prop}
\begin{proof}
Since $\Delta$ is torsion free, $\Delta \to \G_{\Z}(\Delta)$ is an isomorphism. Let $M:=\G_{\Z}(\Delta) \backslash \G_{\R}(\Delta)$. This is a compact manifold whose fundamental group is isomorphic to $\G_{\Z}(\Delta) \simeq \Delta$ and universal cover is homeomorphic to $\G_{\R}(\Delta)$, thus contractible via the exponential map.  Therefore, it is a $K(\Delta, 1)$-space. Its dimension is $\dim M=\dim_{\R}\G_{\R}(\Delta)=r$, hence $H^k(M, \Z)=H^k(\Delta, \Z)=0$ for $k>r$. The manifold $M$ is parallelisable (via the right action of $\G_{\R}(\Delta)$), so it is orientable and $H^r(M, \Z)=\Z$.
\end{proof}

\subsection{Higher Albanese manifolds}\label{higher albanese subsec}

Let $X$ be normal projective variety over $\C$ and $\Gamma:=\pi_1(X,x)$. We denote $\G^s_{R}(X,x):= \G^s_R(\Gamma)$ where $R$ is either $\Z$ or a field of characteristic zero. We denote by $\g^s_{\Bbbk}(X,x)$ the Lie algebra of $\G^s_{\Bbbk}(X,x)$.

The following theorem was proved several times by different authors using different techniques (\cite{Morg, Hain, Simp}). We follow the Hain's approach \cite{Hain}. 

\begin{thrm}[Morgan, Hain, Simpson]
For each $s$ the $\Q$-vector space $\g^s_{\Q}(X,x)$ carries a polarisable mxed $\Q$-Hodge structure $(\g^s_{\Q}(X, x), W_{\bullet}, F^{\bullet})$ such that the following holds:
\begin{enumerate}
\item $W_{-1}\g^s_{\Q}(X, x)=\g^s_{\Q}(X,x)$ (i.e. it is concentrated in negative weights);
\item\label{lie} the Lie bracket $[-,-] \colon \Lambda^2\g^s_{\Q}(X,x) \to \g^s_{\Q}(X, x)$ is a morphism of mixed Hodge structures;
\item if $(X,x) \to (Y,y)$ is a morphism between marked normal varieties,the induced map $\g^s_{\Q}(X,x) \to g^s_{\Q}(Y, y)$ is a morphism of mixed Hodge structures;
\item for $s=1$ it is dual to the Deligne's mixed Hodge structure on cohomology under identification
\[
\g^1_{\Q}(X,x)\simeq H_1(\pi_1(X), \Q)\simeq H_1(X, \Q) \simeq H^1(X, \Q)^*.
\]
\end{enumerate}
\end{thrm}

Remark that item (\ref{lie}) implies that the projections $\g^s_{\Q}(X,x) \to \g^{s-1}_{\Q}(X,x)$ are morphisms of mixed Hodge structures. 

It also implies that $[F^p\g^s_{\C}(X,x), F^q\g^s_{\C}(X,x)] \subseteq F^{p+q}\g^s_{\C}(X,x)$. In particular, $F^0\g^s_{\C}(X,x)$ is a Lie subalgebra and $F^0\G^s(X,x):=\exp(F^0\g^s_{\C}(X,x))\subset \G^s_{\C}(X,x)$ is a closed subgroup.

\begin{df}[Hain - Zucker, \cite{HZ}]
The \emph{$s$-th Albanese manifold} of $X$ is defined as
\[
\Alb^s(X, x):=\G^s_{\Z}(X,x) \backslash \G^s_{\C}(X,x)/F^0\G^s(X,x).
\]
\end{df}

$\Alb^s(X)$ is always a smooth complex manifold. 

A morphism of marked algebraic varieties $f \colon (X,x) \to (Y,y)$ induces a holomorphic map $\Alb^s(f) \colon \Alb^s(X,x) \to \Alb^s(Y, y)$.

For $s=1$ one recovers the classical (quasi-)Albanese variety.

\begin{rmk}
Although the mixed Hodge structure on $\g^s_{\Q}(X,x)$ does depend on the marked point $x \in X$ the higher Albanese manifolds $\Alb^s(X,x)$ and $\Alb^s(X, x')$ are canonically biholomorphic for different points $x$ and $x'$. Starting from now we will always omit the marked point from the notation and write simply $\pi_1(X), \ \g^s_{\Q}(X), \ \Alb^s(X)$ etc.
\end{rmk}

The projections $\G^s_{\C}(X) \to \G^{s-1}_{\C}(X)$ descend to holomorphic maps $p^s \colon \Alb^s(X) \to \Alb^{s-1}(X)$. Each $p^s$ is a holomorphic principal $C^s$-bundle, where $C^s$ is a connected commutative complex Lie group.

This group can be described explicitly as follows. Let 
\[
\mathcal{Z}^s_{\C}(X):=\ker(\G^s_{\C}(X) \to \G^{s-1}_{\C}(X)).
\]
It is the additive group of a complex vector space and  $\mathcal{Z}_{\Z}^s(X):=\mathcal{Z}^s_{\C}(X) \cap \G^s_{\Z}(X)$ is a Zariski dense lattice in it contained in $\mathcal{Z}^s_{\R}(X) \subset \mathcal{Z}^s_{\C}(X)$. The Lie algebra $\z_{\Q}^s(X)=\operatorname{Lie}(\mathcal{Z}_{\Q}^s(X))$ is a Hodge substructure of $\g^s_{\Q}(X)$ and $F^0\z^s=\z^s_{\C} \cap F^0\g^s$. The exponential map $\exp \colon \z^s \to \mathcal{Z}^s$ is an isomorphism and $F^{\bullet}\mathcal{Z}^s_{\C}=\exp(F^{\bullet}\z^s)$ defines a Hodge filtration on the vector space $\mathcal{Z}^s_{\C}(X)$. The group $C^s$ can be described as
\[
C^s=\mathcal{Z}_{\Z}^s(X) \backslash \mathcal{Z}_{\C}^s(X)/F^0\mathcal{Z}^s,
\]
i.e. it is the Jacobian of the (mixed) Hodge structure $(\mathcal{Z}^s(X), F^{\bullet}\mathcal{Z}^s(X))$. 

Let us say a few words about the topology of higher Albanese manifolds.

Its fundamental group isomorphic to $\G^s_{\Z}(X,x)$ and the universal cover $\widetilde{\Alb^s(X)}=\G^s_{\C}(X,x)/F^0\G^s(X,x)$ is contractible. Hence, it has homotopy type $K(\G^s_{\Z}(X,x), 1)$.

\begin{prop}\label{albr homotopy equiv prop}
Let $\Alb^s_{\R}(X):=\G^s_{\Z}(X) \backslash \G^s_{\R}(X)$. This is a smooth compact manifold that admits a smooth embedding $\Alb^s_{\R}(X) \hookrightarrow \Alb^s(X)$ which is a homotopy equivalence. 
\end{prop}
\begin{proof}
$\Alb^s_{\R}(X)$ is a smooth manifold since $\G^s_{\Z}(X)$ is torsion free. It is compact by item (\ref{real completion}) of Proposition \ref{completion basic prop}.

Since the mixed Hodge structure on $\g^s(X)$ is concentrated in negative weights, one has $F^0\g^s(X) \cap \overline{F^0\g^s(X)}=0$, and thus $\g^s_{\R}(X) \cap F^0\g^s(X)=0$. Therefore, the composition 
\[
\g^s_{\R}(X) \hookrightarrow \g^s_{\C}(X) \to \g^s_{\C}(X)/F^0\g^s(X)
\]
is injective.  This implies the injectivity of the map
\[
\G^s_{\R}(X) \to \G^s_{\C}(X) \to \G^s_{\C}(X)/F^0\G^s(X).
\]
Since $\G^s_{\Z}(X) \subset \G^s_{\R}(X)$, it descends to a smooth embedding $\Alb^s_{\R}(X) \hookrightarrow \Alb^s(X)$. Both $\Alb^s_{\R}(X)$ and $\Alb^s(X)$ are aspherical and the embedding induces an isomorphism on the fundamental groups, thus a homotopy equivalence. The last assertion follows from Proposition \ref{cohomological dimension prop} applied to $\G^s_{\Z}(X)$.
\end{proof}

\subsection{Higher Albanese maps}

\begin{thrm}[Hain - Zucker, \cite{HZ}]\label{maps exist thrm}
For a normal quasi-projective variety $X$ and a natural number $s$ there exists a holomorphic map $\alb^s \colon X \to \Alb^s(X)$ such that 
\[
(\alb^s)_* \colon \pi_1(X) \to \pi_1(\Alb^s(X))=\G^s_{\Z}(X)
\]
coincides with the natural projecton $\pi_1(X) \xrightarrow{\varpi^s} \Gamma^s \xrightarrow{/(\mathrm{torsion})} \G^s_{\Z}(X)$.

The diagram of holomorphic maps
\begin{equation}\label{diagram}
\xymatrix{
& \vdots \ar[d]^{p^{s+1}}\\
&\Alb^s(X) \ar[d]^{p^s} \\
& \vdots \ar[d]^{p^3}\\
& \Alb^2(X) \ar[d]^{p^2}\\
X^{\an} \ar[r]_{\alb} \ar[ru]^{\alb^2} \ar[ruuu]^{\alb^s}& \Alb(X).
}
\end{equation}
commutes. If $f \colon (X,x) \to (Y,y)$ is a morphism of marked normal varieties, it induces a commutative diagram
\[
\xymatrix{
X \ar[r]^f \ar[d]_{\alb^s_X} & Y \ar[d]^{\alb^s_Y}\\
\Alb^s(X) \ar[r]_{\Alb^s(f)} & \Alb^s(Y).
}
\]
\end{thrm}

Higher Albanese manifolds are merely complex manifolds and usually do not admit any structure of algebraic varieties. Nevertheless, they possess some canonical definable structure.  We assume that the reader is familiar with definable complex analytic geometry in the spirit of \cite{BBT}. In any case, the only corollary that is important in the proof of the main theorem is Proposition \ref{closure def prop} below.

\begin{thrm}[\cite{Rog}]\label{definability thrm}
The manifolds $\Alb^s(X)$ can be endowed with structure of $\R_{\alg}$-definable smooth complex analytic spaces in such a way that the maps $p^s$ and $\Alb^s(f)$ are $\R_{\alg}$-definable and  $\alb^s \colon X^{\odef} \to \Alb^s(X)$ are $\R_{\an, \exp}$-definable.
\end{thrm}

\begin{prop}\label{closure def prop}
Let $X$ be normal algebraic variety. Let $Y=\overline{\alb^s(X)}$ be the closure of the image in $\Alb^s(X)$. Then  $Y$ is a closed complex analytic subvariety of $\Alb^s(X)$ of the same dimension as $\alb^s(X)$.
\end{prop}

Proposition \ref{closure def prop} follows from Theorem \ref{definability thrm} and Peterzil-Starchenko's version of definable Remmert-Stein Theorem, \cite[Theorem 4.13]{PS}. The fact that $\dim Y=\dim \alb^s(X)$ is a consequence of dimension theory of definable sets, \cite[Chapter IV]{VdD}

Higher Albanese maps are subject to one more restriction that comes from Griffiths transversality.

\begin{prop}\label{horizontality prop}
Let $\widetilde{\Alb^s(X)}=\G^s_{\C}(X)/F^0\G^s(X)$ be the universal cover of $\Alb^s(X)$. Consider the left-invariant subbundle $\widetilde{\mathcal{F}}$  of the tangent bundle of $\widetilde{\Alb^s(X)}$ induced by the subspace $F^1\g^s/F^0\g^s \subseteq \g^s_{\C}/F^0\g^s$. It descends to a holomorphic subbundle $\mathcal{F} \subseteq T\Alb^s(X)$. Then the image of $\alb^s$ is tangent to $\mathcal{F}$. In particular, 
\[
\dim \alb^s(X) \le \rk \mathcal{F}=\dim F^1\g^s-\dim F^0\g^s.
\]
\end{prop}
\begin{proof}
There exists an admissible polarised $\Z$-variation of mixed Hodge structures $\mathbb{V}$ on $X$ such that the period map $\Phi_{\mathbb{V}} \colon X \to M$ factors as
\[
\xymatrix{
X \ar[rr]^{\Phi_{\mathbb{V}}} \ar[rd]_{\alb^s} && M\\
& \Alb^s(X) \ar[ur]_{\Psi}
}
\]
where $M$ is a certain mixed Hodge variety (\cite[Corollary 5.20]{HZ}, cf. \cite[Theorem 6.1]{Rog}). Moreover, the map $\Psi$ is a finite cover on its image and it sends $\mathcal{F}$ to the horizontal distribution on $M$ (\cite[Proposition 6.2.]{Rog}). At the same time, $\Phi_{\mathbb{V}}$ is subject to Griffiths transversality.
\end{proof}

Finally, let us have a look on the topology of higher Albanese maps.

\begin{prop}\label{cohomology prop}
For a topological space $S$ let $\tau_S \colon H^{\bullet}(\pi_1(S), \Z) \to H^{\bullet}(S, \Z)$ denote the natural morphism induced by the map $S \to K(\pi_1(S), \Z)$. Let $\mu^s_{\Z} \colon \pi_1(X) \to \G^s_{\Z}(X)$ be the natural epimorphism, which is restriction $\mu^s_{\Q}$ on its image. There is a commutative diagram of morphisms of graded rings
\[
\xymatrix{
H^{\bullet}(\G^s_{\Z}(X), \Z) \ar[d]^{\tau_{\Alb^s(X)}} \ar[rr]^{(\mu^s_{\Z})^*} && H^{\bullet}(\pi_1(X), \Z) \ar[d]^{\tau_X}  \\
H^{\bullet}(\Alb^s(X), \Z) \ar[rr]_{(\alb^s)^*}&& H^{\bullet}(X,\Z)
}
\]
\end{prop}
\begin{proof}
Follows from and Proposition \ref{albr homotopy equiv prop} and  Theorem \ref{maps exist thrm}.
\end{proof}

\section{Main Theorem}
\subsection{q-completeness}

Let $M$ be a complex manifold of $\dim_{\C}M=n$ and $\phi \colon M \to \R$ a smooth function. Recall that $\phi$ is called \emph{(strictly) $q$-pseudoconvex} if its Levi form $\i\di\dibar\phi$ has at least $n-q$ non-negative (strictly positive) eigenvalues at each point\footnote{In some literature, including the  seminal paper \cite{AG}, different convention is used and such function would be called (strictly) $q+1$-pseudoconvex}. 

One says that $M$ is \emph{$q$-complete} if it admits an exhausting strictly $q$-pseudoconvex function. 

If $\phi$ is a (strictly) $q$-pseudoconvex function and $N \subseteq M$ a closed holomorphic submanifold then $\phi|_N$ is (strictly) $q$-pseudoconvex. This motivates the following definition.

\begin{df}
Let $S$ be a complex analytic space. A continuous function $\phi \colon S \to \R$ is \emph{(strictly) $q$-pseudoconvex} if for each point $s \in S$ there exists an open neighbourhood $U(s) \subseteq S$, a closed holomorphic embedding $U(s) \hookrightarrow \mathbb{B}^N_{\C}$ into a unit ball and a smooth (strictly) $q$-pseudoconvex function $\widetilde{\phi} \colon \mathbb{B}_{\C}^N \to \R$ such that $\widetilde{\phi}|_{U(s)}=\phi$. A space $S$ is \emph{$q$-complete} if it admits an exhausting strictly $q$-pseudoconvex function. 
\end{df}

Remark that if $M$ is a $q$-complete complex manifold and $S \subseteq M$ is a closed complex analytic subvariety, it is automatically $q$-complete (the property of being exhausting is preserved by restriction on a closed subset). 

A space is $0$-complete if and only if it is Stein. A compact analytic space  $S$ is $q$-complete for $q \ge \dim_{\C}S$.

We will use the following two results on $q$-complete spaces. The first one is a theorem of Hamm that connects $q$-completeness to the homotopy type (\cite{Hamm}).

\begin{thrm}[Hamm]\label{Hamm thrm}
Let $S$ be a complex analytic space, $\dim_{\C} S=n$. Suppose that $S$ is $q$-complete. Then $S$ has homotopy type of a CW-complex of dimension at most $n+q$. In particular, $H^k(S, \Z)=0$ for $k>n+q$.
\end{thrm}

The second is a theorem of Takeuchi on $q$-completeness of holomorphic principal bundles (\cite[Theorem 36]{Tak}).

\begin{thrm}[Takeuchi]\label{Takeuchi thrm}
Let $B$ be a complex manifold and $G$ a connected complex Lie group. Let $P \to B$ be a holomorphic principal $G$-bundle. Suppose that $B$ is $r$-complete and $G$ is $q$-complete. Then $P$ is $(r+q)$-complete.
\end{thrm}

\begin{rmk} In general it is not true that $q$-completeness is additive in holomorphic fiber bundles. A Serre's conjecture saying that a fiber bundle with Stein base and Stein fiber is Stein was disproved by Skoda in \cite{Skoda}.
\end{rmk}

\subsection{$q$-completeness of higher Albanese manifolds}

Now we are going to prove the following lemma.

\begin{lemma}\label{completeness of alb}
Let $X$ be a normal complex algebraic variety and $s>0$. Let $q=\dim F^0\g^s$. Then $\Alb^s(X)$ is $q$-complete.
\end{lemma}

\begin{rmk}
In the case $s=1$ Lemma \ref{completeness of alb} can be illustrated by the following two extreme examples. If $F^0\g^1=0$, i.e. the mixed Hodge structure on $H_1(X,\Q)$ is of Tate type, then $\Alb^1(X) \simeq (\C^{\times})^k$ is Stein. On the opposite, if the Hodge structure on $H_1(X, \Q)$ is pure (e.g. $X$ is projective), then $\dim_{\C} F^0\g^1=\frac{1}{2}b_1(X)=\dim_{\C} \Alb^1(X)$ which is coherent with the fact that $\Alb^1(X)$ is a compact complex manifold.
\end{rmk}

Lemma \ref{completeness of alb} follows from Theorem \ref{Takeuchi thrm} and the following simple Proposition.

\begin{prop}\label{abelian completeness}
Let $V$ be a real vector space, $\dim V=n$, and $\Lambda \subset V$ a lattice in it. Let $V_{\C}=V \o \C$ and $F \subset V_{\C}$ be a complex subspace such that $F \cap \overline{F}=0$. Then $C:=\Lambda \backslash V_{\C}/F$ is $q$-complete, where $q=\dim_{\C} F$.
\end{prop}
\begin{proof}

Let $\pi \colon V_{\C} \to V_{\C}/F$ be the projection. Since $F \cap \overline{F}=0$, the map $\pi$ is injective on $V$. Denote $R:=\pi(V)$. 

Let $Q:=R \cap iR$. This is the maximal complex subspace of $R$. We claim that $\dim_{\C}Q=q$. Indeed, the image of a vector $v \in V$ is contained in $Q$ if and only if
\[
iv=v'+f
\]
for some $v' \in V, \ f \in F$. Therefore $v=\operatorname{Re}(f)$. Since $F \cap \overline{F}=0$, the map 
\[
F \to V, \ f \mapsto \operatorname{Re}(f)
\]
is injective and its image projects isomorphically to $Q$.

Choose a direct complement $L$ of $Q$ in $R$ and let $L^c=L \oplus i L$. Observe that $L \cap iL= R \cap iL=0$ and 
\begin{equation}\label{real split}
V_{\C}/F=R \oplus iL
\end{equation}
as a real vector space. As a complex vector space, $V_{\C}/F$ admits a decomposition into a direct sum of two complex vector spaces
\begin{equation}\label{complex split}
V_{\C}/F= Q \oplus L^c.
\end{equation}
of dimensions $q$ and $n-q$ respectively.

Let $\widetilde{\mu} \colon V_{\C}/F \to iL$ be the linear projection along the decomposition (\ref{real split}). Set $\widetilde{\phi}:=||\widetilde{\mu}||^2$. In the linear coordinates adapted to the decomposition (\ref{complex split}) this function is written as
\[
\widetilde{\phi}(z_1, \ldots, z_q, w_1, \ldots, w_{n-q})= \sum_{j=1}^{n-q} (\operatorname{Im} w_j)^2.
\]
A direct computation verifies that $\i\di\dibar\widetilde{\phi}$ has exactly $(n-q)$ positive eigenvalues.

Identify the universal cover $\widetilde{C}$ of $C$ with $V_{\C}/F$. Since $\pi(\Lambda)$ is contained in $R$, the decomposition (\ref{real split}) descends to a splitting
\[
C\simeq (\pi(\Lambda) \backslash R) \times \R^{n-q},
\]
where the projection on the second factor is given by the descend $\mu$ of $\widetilde{\mu}$. The function $\widetilde{\phi}$ descends to a $q$-pseudoconvex function $\phi=||\mu||^2$ on $C$. Since $\pi(\Lambda)$ is a lattice in $R$, the factor $\pi(\Lambda) \backslash R$ is compact and $\phi$ is proper. Since $\phi$ is moreover non-negative, it is exhausting.
\end{proof}

\begin{proof}[Proof of Lemma \ref{completeness of alb}]
Argue by induction in $s$. Recall that $\Alb^s(X)$ is a principal holomorphic $C^s$-bundle over $\Alb^{s-1}(X)$ with
\[
C^s=(\G^s_{\Z}(X) \cap \mathcal{Z}^s) \backslash \mathcal{Z}^s_{\C}/F^0\mathcal{Z}^s.
\]
(this is true even for $s=1$ if one declares $\Alb^0(X)$ to be a point). By Proposition \ref{abelian completeness} the group $C^s$ is $(\dim_{\C} F^0\mathcal{Z}^s)$-complete. At the same time, $\g^s \to \g^{s-1}$ is a morphism of mixed Hodge structures, therefore $\dim_{\C}F^0\mathcal{Z}^s=\dim_{\C}F^0\g^s-\dim_{\C} F^0\g^{s-1}$. Takeuchi's theorem (Theorem \ref{Takeuchi thrm}) finishes the proof.
\end{proof}

\subsection{Proof of the main theorem}

\begin{thrm}\label{main thrm}
Let $X$ be a normal algebraic variety and $s>0$. Then:
\begin{enumerate}
\item\label{first item true} the natural map $\alpha^k \colon H^k(\mathcal{G}^s_{\mathbb{Z}}(X), \mathbb{Z}) \to H^k(X, \mathbb{Z})$ vanishes for $k>\dim F^1\mathfrak{g}^s$;
\item\label{second item true} either $\mathcal{G}^s_{\mathbb{Z}}(X)$ has nilpotency class at most $2$, or there exists $k$ such that $\alpha^k=0$ and $H^k(\mathcal{G}^s_{\mathbb{Z}}(X), \mathbb{Z}) \neq 0$.
\end{enumerate}
\end{thrm}
\begin{proof}

(\ref{first item true}). Consider the map $\alb^s \colon X \to \Alb^s(X)$.  By Proposition \ref{cohomology prop}, it is enough to show that  $(\alb^s)^* \colon H^{k}(\Alb^s(X), \Z) \to H^{k}(X, \Z)$ vanishes. 

Let $Y:=\overline{\alb^s(X)}$ be the closure of the image in the analytic topology. 

By Proposition \ref{closure def prop} it is a closed holomorphic subvariety of $\dim Y=\dim \alb^s(X)$. 

Let $q:=\dim F^0\g^s$ and $r:=\rk \G^s_{\Z}(X)$. Then 
\[
\dim_{\C} \Alb^s(X)= \dim_{\C} \g^s_{\C}- \dim F^0\g^s= r-q 
\]
and $\Alb^s(X)$ is $q$-complete (Lemma \ref{completeness of alb}). Therefore, $Y$ is also $q$-complete and by Hamm's theorem (Theorem \ref{Hamm thrm}) has homotopy type of a CW-complex of a dimension 
\[
\dim_{\C} Y +q < \dim_{\C}(F^1\g^s/F^0\g^s)+q=\dim F^1\g^s-\dim F^0\g^s+\dim F^0\g^s=\dim F^1\g^s.
\]
In particular, $H^k(Y, \Z)=0$ for $k>\dim F^1\g^s$.

The map $\alb^s$ factors through $Y$, thus $(\alb^s)^* \colon H^k(\Alb^s(X), \Z) \to H^k(X, \Z)$ factors through $H^k(Y, \Z)=0$ and vanishes for $k> \dim F^1\g^s$.
\\

(\ref{second item true}).By Propositions \ref{albr homotopy equiv prop} and \ref{cohomological dimension prop} $H^r(\Alb^s(X), \Z)=\Z$ for $r=\rk \G^s_{\Z}(X)$. It is enough to show that $\dim \alb^s(X)<r-q$. But 
\[
\dim \Alb^s(X)=\dim_{\C} \G^s_{\C}(X)-\dim F^0\G^s=r-q,
\]
therefore $\dim \alb^s(X) \le r-q$ and the equality is obtained if and only if $\alb^s$ is dominant. The latter implies that $\G^s_{\Z}(X)$  is of nilpotency class at most $2$ by \cite[Theorem B]{Rog}\footnote{More generally, one can show that $\g^s \neq F^1 \g^s$ unless it is of nilpotency class at most $2$, using Deligne's canonical splitting of the mixed Hodge structure on $\g^s$.}.
\end{proof}

\begin{cor}\label{aspherical cor}
Let $X$ be a normal complex algebraic variety. Suppose that $\pi_1(X)$ is virtually nilpotent and $X$ is aspherical. Then $\pi_1(X)$ is virtually $2$-step nilpotent.
\end{cor}
\begin{proof}
Let $\Gamma \subseteq \pi_1(X)$ be a finite index nilpotent torsion - free subgroup. Let $X' \to X$ be a finite \'etale cover such that $\pi_1(X') \simeq \Gamma$. Observe that $X'$ is also aspherical. Let $s:=\nilp(\Gamma)$. We need to show that $s \le 2$. 

Then  $\pi_1(X') \simeq \G^s_{\Z}(X')$ and 
\[
H^{\bullet}(\G^s_{\Z}(X'), \Z) \xrightarrow{\sim} H^{\bullet}(\pi_1(X'), \Z) \xrightarrow{\sim} H^{\bullet}(X', \Z)
\]
are isomorphisms. If $s > 2$, we come to a contradiction with Theorem \ref{main thrm}. 
\end{proof}

The following Corollary also strengthens earlier results \cite[Theorem 5.43]{HZ} (in the case where $\pi_1(X)$ is not rationally nilpotent and $s >3$) and \cite{Shim2} (in the case where $\dim \G^s_{\C} \le 8$).

\begin{cor}
Let $X$ be a normal complex algebraic variety. Suppose that $\pi_1(X)$ is not rationally abelian or rationally $2$-step nilpotent. Then $\Alb^s(X)$ does not have a homotopy type of a normal algebraic variety for $s>2$.
\end{cor}

\bibliography{references.bib}{}

@BOOK{ABCKT,
  title     = "Fundamental {G}roups of {C}ompact {K}{\"a}hler Manifolds",
  author    = "Amoros, Jaume and Burger, Marc and Corlette, Kevin and Kotschick, Dieter
               and Toledo, Domingo",

  publisher = "American Mathematical Society",
  series    = "Mathematical Surveys and Monographs",
  month     =  mar,
  year      =  1996,
  address   = "Providence, RI"
}

@article{AC,
  title = {The nilpotent quotients of normal quasi-projective varieties with proper quasi-{A}lbanese map},
  volume = {21},
  ISSN = {1558-8602},
  url = {http://dx.doi.org/10.4310/PAMQ.250115040135},
  DOI = {10.4310/pamq.250115040135},
  number = {3},
  journal = {Pure and Applied Mathematics Quarterly},
  publisher = {International Press of Boston},
  author = {Aguilar Aguilar,  Rodolfo and Campana,  Frédéric},
  year = {2025},
  pages = {911–929}
}

@article{AG,
  title = {Théorèmes de finitude pour la cohomologie des espaces complexes},
  volume = {79},
  ISSN = {2102-622X},
  url = {http://dx.doi.org/10.24033/bsmf.1581},
  DOI = {10.24033/bsmf.1581},
  journal = {Bulletin de la Société Mathématique de France},
  publisher = {Societe Mathematique de France},
  author = {Andreotti,  Aldo and Grauert,  Hans},
  year = {1962},
  pages = {193–259}
}

@article{BBT,
  title={o-minimal {G}{A}{G}{A} and a conjecture of {G}riffiths},
  author={Bakker, Benjamin and Brunebarbe, Yohan and Tsimerman, Jacob},
  journal={Inventiones mathematicae},
  volume={232},
  number={1},
  pages={163--228},
  year={2023},
  publisher={Springer}
}

@article{Camp,
  title = {Remarques sur les groupes de {K}\"{a}hler nilpotents},
  volume = {28},
  ISSN = {1873-2151},
  url = {http://dx.doi.org/10.24033/asens.1715},
  DOI = {10.24033/asens.1715},
  number = {3},
  journal = {Annales scientifiques de l’École normale supérieure},
  publisher = {Societe Mathematique de France},
  author = {Campana,  F.},
  year = {1995},
  pages = {307–316}
}

@article{CDHP,
  title   = {Two-step nilpotent monodromy of local systems on special varieties},
  author  = {Cao, Junyan and Deng, Ya and Hacon, Christopher D. and Paun, Mihai},
  journal = {arXiv preprint arXiv:2603.14539},
  year    = {2026}
}

@article{CDY,
  title   = {Hyperbolicity and fundamental groups of complex quasi-projective varieties {(III)}: applications},
  author  = {Cadorel, Benoit and Deng, Ya and Yamanoi, Katsutoshi},
  journal = {arXiv preprint arXiv:2512.20360},
  year    = {2025}
}

@article{Coll,
  title = {The fundamental group of the open symmetric product of a hyperelliptic curve},
  volume = {178},
  ISSN = {1572-9168},
  url = {http://dx.doi.org/10.1007/s10711-014-9998-7},
  DOI = {10.1007/s10711-014-9998-7},
  number = {1},
  journal = {Geometriae Dedicata},
  publisher = {Springer Science and Business Media LLC},
  author = {Collino,  Alberto},
  year = {2014},
  month = Aug,
  pages = {15–19}
}

@article{DGMS,
  author  = {Deligne, Pierre and Griffiths, Phillip and Morgan, John and Sullivan, Dennis},
  title   = {Real homotopy theory of {K}\"ahler manifolds},
  journal = {Inventiones Mathematicae},
  volume  = {29},
  number  = {3},
  pages   = {245--274},
  year    = {1975},
  doi     = {10.1007/BF01389853}
}

@article{GM,
  title   = {An explicit class of {L}agrangian surfaces},
  author  = {Grossi, Paolo and Moretti, Federico},
  journal = {arXiv preprint arXiv:2502.13087},
  year    = {2025}
}

@inproceedings{Hain,
  title={The geometry of the mixed {H}odge structure on the fundamental group},
  author={Hain, Richard},
  booktitle={Proc. Symp. Pure Math},
  volume={46},
  number={2},
  pages={247--282},
  year={1987}
}

@article{Hamm,
author = {Hamm, Helmut A.},
journal = {Journal für die reine und angewandte Mathematik},
pages = {1-9},
title = {Zum {H}omotopietyp q-vollständiger {R}äume},
url = {http://eudml.org/doc/152787},
volume = {364},
year = {1986},
}

@article{Has,
  title = {Minimal models of nilmanifolds},
  volume = {106},
  ISSN = {0002-9939},
  url = {http://dx.doi.org/10.1090/S0002-9939-1989-0946638-X},
  DOI = {10.1090/s0002-9939-1989-0946638-x},
  number = {1},
  journal = {Proceedings of the American Mathematical Society},
  publisher = {American Mathematical Society (AMS)},
  author = {Hasegawa,  Keizo},
  year = {1989},
  month = May,
  pages = {65–71}
}

@article{HZ,
  title={Unipotent variations of mixed {H}odge structure},
  author={Hain, Richard  and Zucker, Steven},
  journal={Inventiones mathematicae},
  volume={88},
  number={1},
  pages={83--124},
  year={1987},
  publisher={Springer-Verlag Berlin/Heidelberg}
}

@article{Mal,
  title={On a class of homogeneous spaces},
  author={Mal'tsev, Anatolii Ivanovich},
  journal={Izvestiya Rossiiskoi Akademii Nauk. Seriya Matematicheskaya},
  volume={13},
  number={1},
  pages={9--32},
  year={1949},
  publisher={Russian Academy of Sciences, Steklov Mathematical Institute of Russian~…}
}

@article{Morg,
  title={The algebraic topology of smooth algebraic varieties},
  author={Morgan, John W},
  journal={Publications Math{\'e}matiques de l'IH{\'E}S},
  volume={48},
  pages={137--204},
  year={1978}
}

@article{Merk,
  title={Grothendieck-{T}eichm{\"u}ller group, operads and graph complexes: a survey},
  author={Merkulov, Sergei},
  journal={arXiv preprint arXiv:1904.13097},
  year={2019}
}

@inbook{PS,
  title = {Complex analytic geometry in a nonstandard setting},
  ISBN = {9780511735226},
  url = {http://dx.doi.org/10.1017/CBO9780511735226.008},
  DOI = {10.1017/cbo9780511735226.008},
  booktitle = {Model Theory with Applications to Algebra and Analysis},
  publisher = {Cambridge University Press},
  author = {Peterzil,  Ya’acov and Starchenko,  Sergei},
  year = {2008},
  month = May,
  pages = {117–166}
}

@article{Quil,
  title={Rational homotopy theory},
  author={Quillen, Daniel},
  journal={Annals of Mathematics},
  volume={90},
  number={2},
  pages={205--295},
  year={1969},
  publisher={JSTOR}
}

@article{Rog,
  title   = {O-minimal geometry of higher {A}lbanese manifolds},
  author  = {Rogov, Vasily},
  journal = {arXiv preprint arXiv:2505.07632},
  year    = {2025}
}

@article{Shim1,
  title   = {On the torsion-free nilpotent fundamental groups of smooth quasi-projective varieties of rank up to seven},
  author  = {Shimoji, Taito},
  journal = {arXiv preprint arXiv:2510.09026},
  year    = {2026}
}

@article{Shim2,
  title   = {Quasi-projective nilmanifolds},
  author  = {Shimoji, Taito},
  journal = {arXiv preprint arXiv:2601.16433},
  year    = {2026}
}

@article{Simp,
  title={Higgs bundles and local systems},
  author={Simpson, Carlos T},
  journal={Publications Math{\'e}matiques de l'IH{\'E}S},
  volume={75},
  pages={5--95},
  year={1992}
}

@article{Skoda,
author = {Skoda, H.},
journal = {Inventiones mathematicae},
pages = {97-108},
title = {Fibrés holomorphiques à base et à fibre de {S}tein},
url = {http://eudml.org/doc/142510},
volume = {43},
year = {1977},
}

@article{Suc,
  title = {Formality and finiteness in rational homotopy theory},
  volume = {10},
  ISSN = {2308-216X},
  url = {http://dx.doi.org/10.4171/EMSS/74},
  DOI = {10.4171/emss/74},
  number = {2},
  journal = {EMS Surveys in Mathematical Sciences},
  publisher = {European Mathematical Society - EMS - Publishing House GmbH},
  author = {Suciu,  Alexander I.},
  year = {2023},
  month = Nov,
  pages = {321–403}
}

@article{SvdV,
  title = {Homotopy groups of pullbacks of varieties},
  volume = {102},
  ISSN = {2152-6842},
  url = {http://dx.doi.org/10.1017/S002776300000043X},
  DOI = {10.1017/s002776300000043x},
  journal = {Nagoya Mathematical Journal},
  publisher = {Cambridge University Press (CUP)},
  author = {Sommese,  Andrew John and Ven,  A. van de},
  year = {1986},
  month = June,
  pages = {79–90}
}

@article{Tak,
  author  = {Takeuchi, Shigeru},
  title   = {On completeness of holomorphic principal bundles},
  journal = {Nagoya Mathematical Journal},
  volume  = {57},
  pages   = {121--138},
  year    = {1974},
  doi     = {10.1017/S0027763000015921}
}

@book{VdD,
  title={Tame topology and o-minimal structures},
  author={Van den Dries, Lou},
  volume={248},
  year={1998},
  publisher={Cambridge university press}
}
\bibliographystyle{alpha}

\end{document}